\documentclass[a4paper,12pt]{article}
\usepackage[text={6.5in,8.5in},centering]{geometry}
\usepackage{graphicx,url,subfig}
\RequirePackage[OT1]{fontenc}
\RequirePackage{amsthm,amsmath,amssymb,amscd}
\RequirePackage[numbers]{natbib}
\RequirePackage[colorlinks,citecolor=blue,urlcolor=blue]{hyperref}
\usepackage{enumerate}
\usepackage{datetime}
\usepackage{etex}
\usepackage[normalem]{ulem} 
\usepackage{soul} 
\makeatother
\numberwithin{equation}{section}
\allowdisplaybreaks

\theoremstyle{plain}
\newtheorem{theorem}{Theorem}[section]

\newtheorem{lemma}[theorem]{Lemma}
\newtheorem{proposition}[theorem]{Proposition}

\theoremstyle{definition}
\newtheorem{definition}[theorem]{Definition}

\newtheorem{remark}[theorem]{Remark}

\newtheorem{example}[theorem]{Example}

\newcommand{\E}{\mathbb{E}}

\newcommand{\ud}{\ensuremath{\mathrm{d}}}

\newcommand{\Floor}[1]{\left\lfloor #1 \right\rfloor}

\newcommand{\Norm}[1]{\left|\left|  #1   \right|\right|}

\newcommand{\Itos}{It\^{o}'s }

\newcommand{\calB}{\mathcal{B}}

\newcommand{\calF}{\mathcal{F}}

\newcommand{\calK}{\mathcal{K}}

\newcommand{\calL}{\mathcal{L}}
\newcommand{\calM}{\mathcal{M}}
\newcommand{\calN}{\mathcal{N}}

\newcommand{\calP}{\mathcal{P}}

\newcommand{\calS}{\mathcal{S}}

\newcommand{\bbN}{\mathbb{N}}
\newcommand{\bbZ}{\mathbb{Z}}

\newcommand{\R}{\mathbb{R}}

\DeclareMathOperator{\Lip}{\mathit{L}}

\DeclareMathOperator{\lip}{\mathit{l}}

\title{Nonlinear stochastic heat equation driven by spatially colored noise: moments and intermittency}

\author{
{\bf Le Chen\footnote{
Research partially supported by a fellowship from Swiss National Science Foundation (P2ELP2\_151796).}}
\quad
and 
\quad {\bf Kunwoo Kim}\footnote{Research supported by the NSF Grant No. 0932078000 while the author was in residence at the Mathematical Sciences Research Institute in Berkeley, California, during the Fall 2015 semester.}
\\[1em]
University of Utah
\date{\vspace{0em}\small \today}
}

\begin{document}
\maketitle
\begin{abstract}
In this paper, we study the stochastic heat equation in the spatial domain $\R^d$ subject to
a Gaussian noise which is white in time and colored in space.
The spatial correlation can be
any symmetric, nonnegative and nonnegative-definite function
that satisfies {\it Dalang's condition}.
We establish the existence and uniqueness of a random field solution starting from measure-valued initial data.
We find the upper and lower bounds for the second moment.
As a first application of these moments bounds,
we find the necessary and sufficient conditions for the solution to have phase transition for the second
moment Lyapunov exponents.
As another application, we prove a localization result for the intermittency fronts.


\end{abstract}

\setlength{\parindent}{1.5em}

\section{Introduction}

In this paper, we will study the following stochastic heat equation,
\begin{align}\label{E:SHE}
\begin{cases}
\displaystyle \left(\frac{\partial }{\partial t} - \frac{\nu}{2}\:
\Delta \right) u(t,x) =  \rho(u(t,x))
\:\dot{M}(t,x),&
x\in \R^d,\; t \in\R_+^*:=\:]0,\infty[\:, \\
\displaystyle \quad u(0,\circ) = \mu(\circ)
\end{cases}
\end{align}
where $\dot{M}$ is a Gaussian noise (white in time and homogeneously colored in space),
$\nu>0$ is the diffusion parameter,
$\rho$ is a globally Lipschitz continuous function satisfying linear growth condition:
\[
\Lip_\rho:= \sup_{z\in\R}\frac{|\rho(z)|}{|z|}<\infty.
\]
The initial data $\mu$ is a deterministic and locally finite (regular) Borel measure.
Informally,
\[
\E\left[\dot{M}(t,x)\dot{M}(s,y)\right] = \delta_0(t-s)f(x-y)
\]
where $\delta_0$ is the Dirac delta measure with unit mass at zero and $f$ is a ``correlation function'' (i.e., a nonnegative, nonnegative definite, and symmetric function that is not identically
zero).
The Fourier transform of $f$ is denoted by $\hat{f}$
\[
\hat{f}(\xi)= \calF f (\xi) =\int_{\R^d}\exp\left(- i \: \xi\cdot x \right)f(x)\ud x.
\]
In general, $\hat{f}$ is again a nonnegative and nonnegative-definite measure, which is usually called  the {\it spectral measure}.
When $\hat{f}$ is genuinely a measure, $\hat{f}(\xi)\ud \xi$ is to be understood as $\hat{f}(\ud \xi)$.
For existence of a random field solution to \eqref{E:SHE},
a necessary condition for the correlation function $f$ is {\it Dalang's} condition \cite{Dalang99,FK09EJP}:
\begin{align}\label{E:Dalang}
\Upsilon(\beta):=(2\pi)^{-d}\int_{\R^d} \frac{\hat{f}(\ud \xi)}{\beta+|\xi|^2}<+\infty \quad \text{for some and hence for all $\beta>0$.}
\end{align}
For the lower bound of the second moment, we will need a stronger assumption on $\rho$:
\begin{align}\label{E:lip}
\lip_\rho:=\inf_{x\in\R} \frac{\rho(x)}{|x|}>0.
\end{align}

There are two main contributions in this paper.
The first one is that the initial data can be Borel measures such as the Dirac delta measure.
The second contribution is that we obtain point-wise moment formulas of the following nature (e.g., $\rho(u)=\lambda u$):
\begin{align}\label{E:MFT}
\E\left[u(t,x)u(t,x')\right]
=\lambda^{-2}\iint_{\R^{2d}} \mu(\ud z)\mu(\ud z') \: \calK(t,x-z,x'-z';z'-z),
\end{align}
for some kernel function $\calK$. When $\rho$ is nonlinear, the above moment formula
turns into lower and upper bounds.
With these moment bounds, we are able to prove several equivalent conditions
for phase transition for the second moment Lyapunov exponents and
we are able to establish existence of intermittency fronts.

\bigskip

In order to state our main results, we need to introduce some notation.
Note that by the Jordan decomposition, the initial measure $\mu$ can be decomposed as $\mu=\mu_+-\mu_-$ where
$\mu_\pm$ are two non-negative Borel measures with disjoint support. Denote $|\mu|:= \mu_++\mu_-$.
The requirement for the initial measure $\mu$ is that
\begin{align}\label{E:J0finite}
\int_{\R^d} e^{-a |x|^2} |\mu|(\ud x)
<+\infty\;, \quad \text{for all $a>0$}\;,
\end{align}
where $|x|^2=x_1^2+\dots+x_d^2$.
Denote this set of measures by $\calM_H\left(\R^d\right)$.
The solution to the homogeneous equation is
\[
J_0(t,x):= \left(\mu* G(t,\circ)\right)(x)= \int_{\R^d}G(t,x-y)\mu(\ud y)
\]
where
\[
G(t,x)= \left(2\pi\nu t\right)^{-d/2} \exp\left(-\frac{|x|^2}{2\nu t}\right).
\]
Then the condition \eqref{E:J0finite} is equivalent to that $J_0(t,x)$ with $\mu$ replaced by $|\mu|$ is finite for all
$t>0$ and $x\in\R^d$.
This is an extension of the work \cite{ChenDalang13Heat} from $\R$ to $\R^d$.
If the initial measure has a bounded density, then {\it Dalang's} condition \eqref{E:Dalang}
is a necessary and sufficient condition for \eqref{E:SHE} to have a random field solution;
see \cite{Dalang99,FK13SHE, FK09EJP, KK13LCA}.
We will show that this statement is still true for all initial measures in $\calM_H(\R^d)$,
provided that
either $\rho(u)=\lambda u$ is linear or
the {\it weak positivity (comparison) principle} holds, namely,
\begin{align}\label{E:WeakCom}
\text{$u(t,x)\ge 0$ a.s. for all $t>0$ and $x\in\R^d$ whenever $\mu\ge 0$.}
\end{align}
Note that the comparison principle is true with some restrictions either on initial data or on the spatial correlation function $f$,
or on both; see \cite{ChenKim14Comp,Mueller91,TessitoreZabczyk98}.
The most general form, i.e., the one under \eqref{E:Dalang} and for $\mu\in\calM_H(\R^d)$, $\mu\ge 0$,
will be a separate project.
The weak comparison principle \eqref{E:WeakCom} is assumed throughout this paper.
Along the proof of existence of solution, we obtain
a moment formula of the type in \eqref{E:MFT}.

\bigskip
Using the moment formula \eqref{E:MFT}, we will study the asymptotic behaviors of the solution.
We first define the {\it upper and lower (moment) Lyapunov exponents of order $p$} ($p\ge 2$) by
\begin{align}\label{E:Lypnv-x}
\overline{m}_p(x) :=\mathop{\lim\sup}_{t\rightarrow+\infty}
\frac{1}{t}\log\E\left(|u(t,x)|^p\right),\quad
\underline{m}_p(x) :=\mathop{\lim\inf}_{t\rightarrow+\infty}
\frac{1}{t}\log\E\left(|u(t,x)|^p\right).
\end{align}
If the initial data are homogeneous (i.e., $\mu(\ud x)=C\ud x$ for some constant $C\in\R$),
then both $\overline{m}_p(x)$ and $\underline{m}_p(x)$ do not depend on $x$.
In this case, a solution is called {\it fully intermittent}
if $\underline{m}_2>0$ and $m_1=0$ by Carmona and Molchanov \cite[Definition
III.1.1]{CarmonaMolchanov94}.
See \cite{khosh1} for a detailed discussion of the meaning of this intermittency property.
For non-homogeneous initial data, we call
a solution fully intermittent if $\inf_{x\in\R^d}\underline{m}_2(x)>0$ and $m_1(x)\equiv 0$.




Foondun and Khoshnevisan proved in \cite[Theorem 1.8]{FK13SHE} and \cite[Theorem 2]{FK14Cor} that
if the correlation function $f$ satisfies Dalang's condition (and some other mild conditions),
and if $\rho(x)$ satisfies \eqref{E:lip}, then for constant initial data,
when $\lip_\rho$ is sufficiently large,
the second moment of the solution to \eqref{E:SHE} has at least exponential growth in time:
\begin{align}\label{E:m2}
\inf_{x\in\R^d} \overline{m}_2(x)>0.
\end{align}
We will show that the condition that $\lip_\rho$ should be sufficiently large
is necessary in certain situations. Moreover, we will strengthen
the statement \eqref{E:m2} into $\inf_{x\in\R^d} \underline{m}_2(x)>0$.

More generally, we will study the necessary and sufficient conditions for {\it the phase transition} of the second moment Lyapunov exponents.
Under \eqref{E:lip}, we call the solution $u(t,x)$ to \eqref{E:SHE}
has a {\it second order phase transition} if there exist two nonnegative constants
$0<\underline{\lambda}_c\le \overline{\lambda}_{c}<\infty$, such that
\begin{align}\label{E:Phase}
\begin{cases}
\sup_{x\in\R^d}\overline{m}_2(x)=0 & \text{if $(\lip_\rho\le)\:\Lip_\rho< \underline{\lambda}_c$,}\\[0.5em]
\inf_{x\in\R^d}\underline{m}_2(x)>0 & \text{if $\overline{\lambda}_c<\lip_\rho \: (\le \Lip_\rho)$.}
\end{cases}
\end{align}
Here the two parameters $\Lip_\rho$ and $\lip_\rho$ play the role of $\lambda$ for the
Anderson model $\rho(u)=\lambda u$.
We will prove that the phase transition happens if and only if
\begin{align}\label{E:iff1}
\Upsilon(0):=\lim_{\beta\rightarrow 0}\Upsilon(\beta)<\infty.
\end{align}
As a consequence, we have the following statements:
\begin{enumerate}
\item No phase transition happens when $d=1$ or $2$.
\item Phase transition happens if and only if
\begin{align}\label{E:iff2}
d\ge 3\quad\text{and}\quad
\int_{\R^d} \frac{f(z)}{|z|^{d-2}}\ud z<\infty.
\end{align}
\item Let $B_t$ be a Brownian motion on $\R^d$ starting from the origin with $\E(|B_t|^2)=\nu t$. Define, for $t>0$,
\begin{align}
\label{E:h1}
\begin{aligned}
k(t)&:=\E(f(B_t))= \int_{\R^d} f(z) G(t,z) \ud z,\\
 h_1(t)&:=\E\left[\int_0^t f(B_s) \ud s\right]=\int_0^t k(s) \ud s.
\end{aligned}
\end{align}
Then \eqref{E:iff1} holds if and only if
\begin{align}\label{E:iff3}
\lim_{t\rightarrow \infty} h_1(t)<\infty.
\end{align}
\end{enumerate}

\begin{remark}
It is well-known that if $f\in \calS(\R^d)$ (the Schwartz test function), then
\[
h_1(\infty)=C_1 \int_{\R^d}\frac{f(x)}{|x|^{d-2}}\ud x = C_2 \int_{\R^d} \frac{\hat{f}(\xi)}{|\xi|^2}\ud \xi=C_3 \Upsilon(0);
\]
see. e.g., \cite[Lemma 2, Chapter 5]{Stein70Singular} for the second equality.
In this case, the equivalence of \eqref{E:iff1}, \eqref{E:iff2}, and \eqref{E:iff3} is clear.
\end{remark}

\begin{remark}
Condition \eqref{E:iff2} sets restrictions on the behaviors of $f$ both at the infinity and around zero.
In particular, when $d\ge 3$, in order to have phase transition,
the local integrability of $f$ around zero is not enough, and
the tails should not be too fat.
Note that the local integrability of $f$ in \eqref{E:iff2} is stronger than that in Dalang's condition \eqref{E:Dalang}.
\end{remark}

We summarize these results in the following theorem.
Recall that $\mu>0$ means that $\mu\ge 0$ (nonnegative) and $\mu\ne 0$ (non-vanishing).

\begin{theorem}\label{T:Phase}
Suppose that the initial data $\mu\in\calM_H(\R)$ is such that $\mu\ge 0$ and
\begin{align}\label{E:InitGrow}
\int_{\R^d} e^{-\beta |x|} \mu(\ud x)<+\infty\quad\text{for all $\beta>0$.}
\end{align}
If $\Upsilon(0)<\infty$, then \eqref{E:Phase} holds.
On the other hand, if $\Upsilon(0)=\infty$ and $\mu> 0$,
then the solution $u(t,x)$ is fully intermittent.
Moreover,
\begin{align}
 \text{Condition \eqref{E:iff1}}
\Longleftrightarrow
\text{Condition \eqref{E:iff2}}
\Longleftrightarrow
\text{Condition \eqref{E:iff3}}.
\end{align}
\end{theorem}

The condition \eqref{E:iff2} tells us that to have phase transition, the behaviors of $f(x)$ both at the origin and at the infinity matter.
If $f$ is radial $f(x)=\tilde{f}(|x|)$, the integral condition in \eqref{E:iff2} reduces to $\int_0^\infty \tilde{f}(r)r\ud r<\infty$.

Let us first have a look of the cases when $f(0)<\infty$.
Examples of such kernel functions include the {\it Ornstein-Uhlenbeck-type kernels}
$f(x)=\exp\left(-c |x|^\alpha\right)$ for $\alpha\in \:]0,2]$ and $c>0$,
the {\it Poisson kernel} $f(x)=(1+|x|^2)^{-(d+1)/2}$ and the {\it Cauchy kernel} $f(x)=\prod_{j=1}^d (1+x_j^2)^{-1}$.
All these examples satisfy the condition \eqref{E:iff2} for $d\ge 3$.
When $\rho(u)=\lambda u$, the above results are proved using the Feynman-Kac representation of the solution
by Nobel \cite[Theorem 9]{Nobel97}, and
its discrete counterpart ($\bbZ^d$ replaced by $\R^d$) has been well studied by Carmona and Molchanov \cite{CarmonaMolchanov94}.

The typical examples for $f(0)=\infty$ are the {\it Riesz kernels} $f(x)=|x|^{-\alpha}$ with $\alpha\in\:]0,2\wedge d[\:$.
They fail the integrability condition in \eqref{E:iff2} due to their fat tails. Hence, there is no phase transition.
Recently, this case has also been studied by Foondun, Liu and Omaba \cite{FLO14MB}.


%

\bigskip
Another application of our moment formula \eqref{E:MFT} is the study of the intermittency front.
Following \cite{ConusKhosh10Farthest}, define the following {\it growth indices}:
\begin{align}
\label{E:SupGrowInd-0}
\underline{\lambda}(p):= &
\sup\left\{\alpha>0: \underset{t\rightarrow \infty}{\lim\inf}\:
\frac{1}{t}\sup_{|x|\ge \alpha t} \log \E\left(|u(t,x)|^p\right) >0
\right\},\\
\label{E:SupGrowInd-1}
 \overline{\lambda}(p) := &
\inf\left\{\alpha>0: \underset{t\rightarrow \infty}{\lim\sup}\:
\frac{1}{t}\sup_{|x|\ge \alpha t} \log \E\left(|u(t,x)|^p\right) <0
\right\}\;.
\end{align}
These quantities characterize the propagation speed of ``high peaks'';
see \cite{ConusKhosh10Farthest,ChenDalang13Heat} for more details.
The higher spatial dimension cases have more geometry than the one space dimensional case.
Here we will give a rough characterization of the locations of the peaks using the space-time cones.
Refined investigations in this direction will be a separate project.

\begin{theorem}\label{T:Front}
Suppose that $\rho$ satisfies \eqref{E:lip} and the initial data $\mu\ge 0$ satisfies that
\begin{align}\label{E:muBeta}
\int_{\R^d}e^{\beta|x|}\mu(\ud x)<+\infty,\quad\text{for some $\beta>0$.}
\end{align}
Then
\begin{align}\label{E:Front}
0\le
\sqrt{\nu \:\theta_*}\:
\le \underline{\lambda}(2)\le
\overline{\lambda}(2) \le
\frac{\sqrt{d}}{2}\left(\nu \beta +\frac{\theta}{\beta}\right)<+\infty,
\end{align}
where the two constants $\theta:=\theta(\nu,\Lip_\rho^2)$ and  $\theta_*:=\theta_*(\nu,\lip_\rho^2)$
are defined as follows
\begin{align}\label{E:Var}
\theta(\nu,\Lip_\rho^2)&:=\inf\left\{\beta>0:  \:\Upsilon\left(2\beta/\nu\right) < \frac{\nu}{2\Lip_\rho^2}\right\},\\
\theta_*(\nu,\lip_\rho^2)&:=\lim_{t\rightarrow\infty} \frac{1}{t}\log\sum_{n=0}^\infty
\left(\left[2\sqrt{3}\right]^{-d}\lip_\rho^2\right)^{n}h_1(t/n)^n.
\end{align}
Moreover, if $\Upsilon(0)=\infty$, then $\theta_*(\nu,\lip_\rho^2)>0$ (strict inequality) for all $\nu>0$ and  $\lip_\rho>0$. Otherwise,
$\theta_*(\nu,\lip_\rho^2)$ is strictly positive when either $\lip_\rho$ is sufficiently large or $\nu$ is sufficiently small.
\end{theorem}

\begin{remark}
When $d=1$, $f=\delta_0$, $\rho(u)=\lambda u$, and the initial measure $\mu\ge 0$ satisfies \eqref{E:muBeta},
it is proved in \cite{ChenDalang13Heat} that
\[
\frac{\lambda^2}{2}\le \underline{\lambda}(2)\le \overline{\lambda}(2)\le \frac{\beta \nu}{2}+ \frac{\lambda^4}{8\nu\beta},
\]
in particular, when $\beta\ge \lambda^2/(2\nu)$, $ \underline{\lambda}(2)= \overline{\lambda}(2)=\lambda^2/2$.
On the other hand, as shown in Example \ref{Eg:WhiteNoise},
$\theta=\nu^{-1}\lambda^4$ and $\theta_*=(6 \pi \nu e^2)^{-1}\lambda^4$.
Hence, by \eqref{E:Front}, when $\beta \ge \lambda^2/\nu$,
\[
0.0847335\: \lambda^2 \approx \frac{\lambda^2}{e\sqrt{6 \pi}} \le \underline{ \lambda}(2)\le \overline{\lambda}(2)\le \lambda^2.
\]
These estimates in \eqref{E:Front} are not as sharp as those in \cite{ChenDalang13Heat}
but they cover more general noises.
\end{remark}

\bigskip
Throughout this paper, $\Norm{\cdot}_p$ denotes the $L^p(\Omega)$ norm.

\bigskip
This paper is organized as follows.
We first study the existence and uniqueness of a random field solution to \eqref{E:SHE}
under rough initial conditions in Section \ref{S:ExUni}.
The phase transition result (Theorem \ref{T:Phase}) is proved in Section \ref{S:Phase}.
The growth indices result (Theorem \ref{T:Front}) is proved in Section \ref{S:Front}.
Finally, some examples are listed in the Appendix.

\section{Existence and uniqueness}\label{S:ExUni}

\subsection{Some prerequisites}\label{S:Pre}

Throughout this subsection let $R(x,y)$ be a non-negative and non-negative
definite kernel in the sense that
\[
\iint_{\R^{2d}} R(x,y)\psi(x)\psi(y)\ud x\ud y\ge 0,\quad \text{for all
$\psi\in C_c^\infty\left(\R^d\right)$}\;,
\]
where
$C_c^\infty\left(\R^d\right)$ be the test functions, i.e.,
functions in $C^\infty\left(\R^d\right)$ with compact support.
Suppose that $R(x,y)$ satisfies the following condition:
\[
\iint_{K\times K} R(x,y)\ud x \ud y <+\infty\;,\qquad \text{for all compact
sets $K\in\R^d$.}
\]
Associated with such $R$, there is a non-negative and locally finite measure,
denoted by $\mu_R$,
over $\R^d$ such that
\[
\mu_R(K) := \iint_{K\times K} R(x,y)\ud x \ud y \;,\quad \text{for all Borel
sets
$K\subseteq \R^d$}\:.
\]

\begin{definition}
A {\em spatially $R$-correlated Gaussian noise that is white in time} is an
$L^2(\Omega)$-valued mean zero Gaussian process
\[
\left\{ F(\psi): \:\psi\in C_c^\infty\left(R^{1+d}\right)\:\right\}\;,
\]
such that
\[
\E\left[F(\psi)F(\phi)\right] =
\int_0^t \ud s \iint_{\R^d} \psi(s,x) R(x,y)\phi\left(s,y\right)\ud x\ud y\:.
\]
\end{definition}
Note that if $R(x,y)=h(x-y)$ for some kernel $h$, then the above definition
reduces to the spatially homogeneous Gaussian noise that is white in time
\cite{Dalang99}.
In particular, if $h(x-y)= \delta_0(x-y)$, then this noise becomes the
space-time white noise and the associated measure $\mu_R$ reduces to the Lebesgue measure on $\R^d$.

%
%
%
%
%

We need some criteria to check whether a random field is predictable.
As in \cite{Dalang99}, we extend $F$ to a $\sigma$-finite $L^2$-valued measure $B \rightarrow F(B)$
defined for bounded Borel sets $B\in  \R_+ \times \R^d$ and then define
\[
M_t(A):=F([0,t]\times A), \quad A\in\calB_b\left(\R^d\right).
\]
Let $(\calF_t,t\ge 0)$ be the filtration given by
\[
\calF_t := \sigma\left(F_s(A):\: 0\le s\le t,
A\in\calB_b\left(\R^d\right)\right)\vee
\calN,\quad t\ge 0,
\]
which is the natural filtration augmented by all $P$-null sets $\calN$ in $\calF$,
where $\calB_b(\R^d)$ is the collection of Borel measurable sets with finite Lebesgue measure.
The family of subsets of $\R_+\times\R^d\times\Omega$, which contains all sets
of the
form $\{0\}\times F_0$ and $]s,t]\times A\times F$, where $F_0\in \calF_0$,
$F\in\calF_s$ for $0\le s<t$ and $A$ is a rectangle in $\R^d$, is called the
class of {\it predictable rectangles}. The $\sigma$-field generated by
the predictable rectangles is called the {\it predictable $\sigma$-field},
which is denoted by $\calP$. Sets in $\calP$ are called {\it predictable sets}.
A random field $X:\Omega\times\R_+\times\R^d\mapsto \R$ is called {\it
predictable} if $X$ is $\calP$-measurable.

For $p\in [2,\infty[\:$, denote $\calP_{p}$ to be the set of all predictable and
$L^2_R\left(\R_+\times\R^d; \: L^p\left(\Omega\right)\right)$
integrable random fields.
More precisely, $f\in\calP_p$ if and only if $f$ is
predictable and
\begin{align}
 \Norm{f}_{M, p}^2 & :
= \iiint_{\R_+^*\times\R^{2d}}R(x,y) \Norm{f(s,x)f\left(s,y\right)}_{\frac{p}{2}}\ud s\ud
x\ud y<+\infty\;,
\end{align}
where $\Norm{\cdot}_p$ denotes the $L^p(\Omega)$-norm.
In particular, if $R(x,y)=\delta_0(x-y)$, then
\[
 \Norm{f}_{M, p}^2
= \iint_{\R_+^*\times\R^{d}} \Norm{f(s,x)}_p^2\ud s\ud x.
\]
Clearly,
\[
\calP_2\supseteq \calP_p \supseteq \calP_q,\quad \text{for $2\le p\le
q<+\infty$.}
\]

The following proposition is useful to check whether a random field belongs
to $\calP_{p}$ or not.

\begin{proposition}\label{P:Pm}
Suppose that for some $t>0$ and $p \in [2,\infty[$, a random field
\[
X=\left\{\: X\left(s,y\right):\: \left(s,y\right)\in \;]0,t[\:\times \R^d\: \right\}
\]
has the following properties:
\begin{itemize}
 \item[(i)] $X$ is adapted, i.e., for all $\left(s,y\right)\in\;]0,t[\:\times \R^d$, $X\left(s,y\right)$
is
$\calF_s$-measurable;
 \item[(ii)] $X$ is jointly measurable with respect to
$\calB\left(\R_+\times \R^d\right)\times\calF$;
 \item[(iii)] $\Norm{X}_{M, p}<+\infty$.
\end{itemize}
Then $X(\cdot,\circ)\: 1_{]0,t[}(\cdot)$ belongs to $\calP_p$.
\end{proposition}

This proposition is an extension of Dalang \& Frangos's result in \cite[Proposition 2]{DalangFrangos98}
in the two senses: (1) the second moment of $X$ can blow up at $s=0$ or $s=t$, which is the case, e.g., when the initial data is the Dirac delta measure;
(2) the condition that $X$ is $L^2(\Omega)$-continuous has been removed.
The proof of this proposition follows essentially the same arguments as
the proof of the the case where $d=1$ and the noise is white in both space and time variables;
see \cite[Proposition 3.1]{ChenDalang13Heat}.

\subsection{Statement of the result}
We formally write the SPDE \eqref{E:SHE} in the integral form
\begin{align}\label{E:WalshSI}
u(t,x)= J_0(t,x)+ I(t,x)
\end{align}
where
\[
I(t,x):= \iint_{[0,t]\times \R^d} G(t-s,x-y) \rho(u(s,y))M(\ud s,\ud y).
\]
The above stochastic integral is understood in the sense of Walsh \cite{Dalang99,Walsh}.

\begin{definition}\label{D:Solution}
A process $u=\left(u(t,x),\:(t,x)\in\R_+^*\times\R^d \right)$  is called a {\it
random field solution} to \eqref{E:SHE} if
\begin{enumerate}[(1)]
 \item $u$ is adapted, i.e., for all $(t,x)\in\R_+^*\times\R^d$, $u(t,x)$ is
$\calF_t$-measurable;
\item $u$ is jointly measurable with respect to
$\calB\left(\R_+^*\times\R^d\right)\times\calF$;
\item $\Norm{I(t,x)}_2<+\infty$ for all $(t,x)\in\R_+^*\times\R^d$;
\item  The function $(t,x)\mapsto I(t,x)$ mapping $\R_+^*\times\R^d$ into
$L^2(\Omega)$ is continuous;
\item $u$ satisfies \eqref{E:WalshSI} a.s.,
for all $(t,x)\in\R_+^*\times\R^d$.
\end{enumerate}
\end{definition}

Denote
\[
J_1(t,x,x'):=J_0(t,x)J_0(t,x')
\]
and $g(t,x,x'):= \E\left[u(t,x)u(t,x')\right]$.
Then by \Itos isometry, $g$ satisfies the following integral equation (for $\rho(u)=\lambda u$)
\begin{align}\label{E:Recursion}
\begin{aligned}
g(t,x,x')= J_1(t,x,x')+\lambda^2 &\int_0^t\ud s \iint_{\R^{2d}}\ud z\ud z'
\; g(s,z,z')\\
\times &G(t-s,x-z)G(t-s,x'-z') f(z-z').
\end{aligned}
\end{align}
Replacing the function $g$ on the r.h.s. of \eqref{E:Recursion} by \eqref{E:Recursion} itself repeatedly
suggests the following definitions.
For $h,w:\R_+\times\R^{3d}\mapsto\R$, define the operation ``$\rhd$'', which depends on $f$, as follows
\begin{equation}
\label{E:RHD-BACK}
\begin{aligned}
\left(h \rhd w\right)(t,x,x';y):=\int_0^t \ud s \iint_{\R^{2d}}\ud z \ud z'\; h(t-s,x-z,x'-z';y-(z-z'))&  \\
\qquad \times  w(s,z,z';y)\: f(y-(z-z'))& .
\end{aligned}
\end{equation}
By change of variables,
\begin{equation}
\label{E:RHD-FOR}
\begin{aligned}
\left(h \rhd w\right)(t,x,x';y):=\int_0^t \ud s \iint_{\R^{2d}}\ud z \ud z'\: h(s,z,z';y-[(x-z)-(x'-z')])&\\
\times  w(t-s,x-z,x'-z';y)\: f(y-[(x-z)-(x'-z')])&.
\end{aligned}
\end{equation}
Note that for general $f$, this convolution-type operator is not symmetric,
$h\rhd w \ne w\rhd h$, except for some special cases, such as,
$f\equiv 1$ or $f=\delta_0$. Operators of this type have been studied in Chen's thesis \cite[Chapter 3]{LeChen13Thesis} \footnote{The operator in \cite[Chapter 3]{LeChen13Thesis} is more general.
Indeed, by taking the spatial dimension to be $2d$, and $\theta^2(t,x)=f(\hat{x}-\hat{x}')$ where
$x=(\hat{x},\hat{x}')$ with $\hat{x}, \hat{x}'\in\R^d$, one reduces the operator in \cite[Chapter 3]{LeChen13Thesis} to the current operator.}.
Some calculations show that by introducing the additional variable $y$, this operator becomes associative, i.e.,
for $h,w,v : \R_+\times\R^{3d}\mapsto \R$,
\[
\left(\left(h\rhd w\right)\rhd v\right)(t,x,x';y)
=\left(h\rhd \left(w\rhd v\right)\right)(t,x,x';y);
\]
See Lemma \ref{L:Associative} below.
We will use the following convention: If $h$ is a function from $\R_+\times\R^{2d}$ to $\R$,
when applying the operation $\rhd$ to $h$, it is meant for $\hat{h}(t,x,x';y):=h(t,x,x')$.

\bigskip
For $t>0$ and  $x,x',y\in\R^d$, define recursively:
\[
\calL_{0}(t,x,x';y):= G(t,x)G(t,x')
\]
and for $n\ge 1$,
\[
\calL_{n}(t,x,x';y):= \left(\calL_{0}\rhd \calL_{n-1}\right)(t,x,x';y).
\]
For $\lambda\in\R$, define formally
\begin{align*}
\calK_{\lambda}(t,x,x';y):= \sum_{n=0}^\infty \lambda^{2(n+1)}\calL_{n}(t,x,x';y).
\end{align*}
The convergence of the above series is proved in Lemma \ref{L:calK} below.
We will use the following convention for $\calK_{\lambda}$:
\begin{align}
\label{E:Conv}
\calK:= \calK_{\lambda} \qquad \overline{\calK}:=\calK_{\Lip_\rho}\qquad
\underline{\calK} := \calK_{\lip_\rho}.
\end{align}
Using these notation and conventions, we see that \eqref{E:Recursion} can be written in the following way:
\begin{align}\label{E:Rec}
g(t,x,x') &= J_1(t,x,x')+ \lambda^2 \left(\calL_{0}\rhd g \right)(t,x,x';0),
\end{align}
which suggests that
\[
g(t,x,x')   =
J_1(t,x,x') + \left(\calK\rhd J_1\right)(t,x,x';0).
\]

\bigskip
\begin{theorem}\label{T:ExUni}
For any $\mu\in\calM_H(\R^d)$, the SPDE \eqref{E:SHE} has a unique (in the sense of versions) random field solution
$\left\{u(t,x): t>0,x \in\R^d\right\}$ starting from $\mu$.
This solution is $L^2(\Omega)$-continuous.
Moreover, the following moment estimates are true:\\
(1)If $\rho(u)=\lambda u$, then the two-point correlation function is equal to
\begin{align}\label{E:SecMom}
 \E\left[u(t,x)u(t,x')\right]
&= J_1(t,x,x') + \left( \calK \rhd J_1\right) (t,x,x';0)\\
&= \lambda^{-2}\iint_{\R^{2d}} \mu(\ud z)\mu(\ud z')\: \calK(t,x-z,x'-z';z'-z).
\label{E:SecMom2}
\end{align}
(2) If $|\rho(x)| \le \Lip_\rho |x|$ for all $x\in\R$ with $\Lip_\rho>0$ and if $\mu\ge 0$, then
\begin{align}\label{E:SecMom-Up}
\E\left[u(t,x)u(t,x')\right] \le
 \Lip_\rho^{-2}\iint_{\R^{2d}} \mu(\ud z)\mu(\ud z')\: \overline{\calK}(t,x-z,x'-z';z'-z).
\end{align}
\noindent
(3) If $\rho$ satisfies \eqref{E:lip}, then
\begin{align}\label{E:SecMom-Lower}
\E\left[\left|u(t,x)u(t,x')\right|\right] \ge
 \lip_\rho^{-2}\iint_{\R^{2d}} \mu(\ud z)\mu(\ud z')\: \underline{\calK}(t,x-z,x'-z';z'-z).
\end{align}
\end{theorem}

Note that the condition in part (2) is true if the weak comparison principle holds and $\mu\ge 0$.
In a recent paper \cite{CHN15TwoPoint}, an explicit expression for this kernel function $\calK$ is
obtained when $d=1$ and $f(x)=\delta_0(x)$.

%


\subsection{Proof of Theorem \ref{T:ExUni}}
We will first prove some results.
Recall the definition of the function $k(t)$ in \eqref{E:h1}.
By the Fourier transform, this function $k(t)$ can also be rewritten in the following form
\begin{align}\label{E:k-F}
k(t)=(2\pi)^{-d} \int_{\R^d}  \hat{f}(\ud \xi)\exp\left(-\frac{\nu t}{2}|\xi|^2\right),
\end{align}
from which one can see that $t\mapsto k(t)$ is a nonincreasing function.
For $t\ge 0$ and $y\in\R^d$, define $h_0(t,y):=1$ and for $n\ge 1$,
\begin{align}\label{E:Def-hny}
 h_{n}(t,y) :=\int_0^t \ud s \: h_{n-1}(s,y) k(t-s)T_{\nu/4}(t-s,y),
\end{align}
where
\[
T_\nu(t,x):=\exp\left(-\frac{|x|^2}{\nu t}\right).
\]
Define
\begin{align}
H_\nu(t,y;\gamma) := \sum_{n=0}^\infty \gamma^n h_{n}(t,y).
\end{align}
We will use the convention that
\[
h_n(t):=h_n(t,0)\quad\text{and}\quad H_\nu(t;\gamma) :=H_\nu(t,0;\gamma).
\]
%

\begin{lemma} \label{L:EstHt}
For all $t\ge 0$ and $\gamma\ge 0$,
\[
H_\nu(t;\gamma)\le \exp\left(\theta t\right),
\]
where this constant $\theta$ can be chosen as
\begin{align}\label{E:Var}
\theta:=\theta(\nu,\gamma)=\inf\left\{\beta>0:  \:\Upsilon\left(2\beta/\nu\right) < \frac{\nu}{2\gamma}\right\}.
\end{align}
Moreover, if $\Upsilon(0)<\infty$ and $\gamma <2\Upsilon(0)/\nu$, then
\[
H_\nu(t;\gamma)\le \frac{\nu}{\nu-2 \gamma \Upsilon(0)} \quad\text{for all $t\ge 0$.}
\]
\end{lemma}
\begin{proof}
Notice that for $\beta>0$,
\[
\int_{\R_+} e^{-\beta t} h_n(t)\ud t
=\frac{1}{\beta}\left(\int_{\R_+}e^{-\beta t} k(t) \ud t\right)^n
=\frac{1}{\beta}\left[\frac{2}{\nu}\: \Upsilon\left(\frac{2\beta}{\nu}\right)\right]^n.
\]
Because
$\Upsilon(\beta)\rightarrow 0$ as $\beta\rightarrow\infty$, by increasing $\beta$,
we can make sure that $2\nu^{-1}\Upsilon\left(2\beta/\nu\right)\gamma <1$.
The smallest $\beta$ that satisfies \eqref{E:Var} gives the constant $\theta$.
When $\Upsilon(0)<\infty$, notice that
\begin{align}\label{E:h1Upsilon}
\lim_{t\rightarrow\infty}h_{1}(t) =\lim_{\beta\rightarrow 0_+} \frac{2}{\nu} \Upsilon(2\beta/\nu).
\end{align}
Hence, by the induction, $h_n(t)\le \left[2 \nu^{-1} \Upsilon(0)\right]^n$ for all $n\ge 0$.
This completes the proof of Lemma \ref{L:EstHt}.
\end{proof}


Even though the integrand in the definition of $h_n$ is positive, due to the presence of $t$ in the integrand,
the following result is nontrivial (considering, e.g., $\int_0^t (s(t-s))^{-2/3}\ud s=C t^{-1/3}$).

\begin{lemma}\label{L:hnInc}
For $n\ge 0$ and $y\in\R^d$, all functions $t\in \R_+\mapsto h_n(t,y)$ are nondecreasing.
\end{lemma}
\begin{proof}
Fix $y\in\R^d$. The case $n=0$ is true by definition. Suppose that it is true for $n$. For all
$\epsilon\ge 0$, by the induction assumption,
\begin{align*}
h_{n+1}(t+\epsilon,y)&=\int_0^{t+\epsilon} \ud s \: h_n(t+\epsilon-s,y) k(s)T_{\nu/4}(s,y)\\
&\ge \int_0^{t} \ud s \: h_n(t+\epsilon-s,y) k(s)T_{\nu/4}(s,y)\\
&\ge \int_0^{t} \ud s \: h_n(t-s,y) k(s)T_{\nu/4}(s,y)=h_{n+1}(t,y).
\end{align*}
This proves the lemma.
\end{proof}

Recall the convention \eqref{E:Conv} for $\calL_{n}$.
\begin{lemma}\label{L:calK}
Suppose that the correlation function  $f$ satisfies Dalang's condition \eqref{E:Dalang}.
Then for all $n\ge 1$, $t\ge 0$, $x,x', y\in\R^d$,
\begin{align}\label{E:calL-hUp}
\calL_n(t,x,x';y) &\le 2^n \: G(t,x)\: G(t,x') h_n(t),\\
\calL_n(t,x,x';y)&\ge  (2\sqrt{3}\:)^{-nd}G(t,x) G(t,x')T_\nu(t,x-x')
 \: h_n\left(t/2,y\right),
\label{E:calL-hLow}
\end{align}
Hence,
\begin{align}\label{E:CalK-Up}
\calK_\lambda(t,x,x';y)&\le
\calL_0(t,x,x') H_\nu\left(t;2\lambda^2\right) ,\\
\label{E:CalK-Low}
\calK_\lambda(t,x,x';y)&\ge
\calL_0(t,x,x') T_\nu(t,x-x')H_\nu(t/2,y;(2\sqrt{3}\:)^{-d}\lambda^2).
\end{align}
\end{lemma}
\begin{proof}
By definition,
\begin{align*}
\calL_1(t,x,x';y) = \int_0^t\ud s\iint_{\R^{2d}}\ud z\ud z' \: &G(s,z) G(s,z') f(y-(z-z'))\\
\times & G(t-s,x-z) G(t-s,x'-z').
\end{align*}
Notice that (see \cite[Lemma 5.4]{ChenDalang13Heat})
\[
G(s,z)G(t-s,x-z) =
G\left(\frac{s(t-s)}{t},z-\frac{s}{t}x\right)G(t,x)
\]
and similar for the other pair.
So
\begin{align}
\label{E:GGf}
\begin{aligned}
\calL_1(t,x,x';y) =& \quad G(t,x)G(t,x')\int_0^t\ud
s \iint_{\R^{2d}}\ud z \ud z' \: f(y-(z-z')) \\
&\times G\left(\frac{s(t-s)}{t},z-\frac{s}{t} x\right)
G\left(\frac{s(t-s)}{t},z'-\frac{s}{t}x'\right).
\end{aligned}
\end{align}
Because
\[
\calF[G(t,\circ)](\xi) = \exp\left(-\frac{\nu t}{2} |\xi|^2 \right),
\]
the double integral over $\ud z\ud z'$ in \eqref{E:GGf} is equal to
\begin{multline}\label{E:FhatG}
(2\pi)^{-d}\int_{\R^d}\: \hat{f}(\ud \xi ) \exp\left(i \left(y-\frac{s}{t}(x-x')\right)\cdot \xi
-\frac{\nu s(t-s)}{t}|\xi|^2\right)\\
=\int_{\R^d}\ud z\; f(z) G\left(\frac{2s(t-s)}{t},z+y-\frac{s}{t}(x-x')\right).
\end{multline}

Now let us prove \eqref{E:calL-hUp}. From \eqref{E:GGf} and \eqref{E:FhatG}, it is clear that
\begin{align}\label{E:L1}
\calL_1(t,x,x';y)\le (2\pi)^{-d} G(t,x)G(t,x')
\int_0^t \ud s \int_{\R^d}  \hat{f}(\ud \xi)\exp\left(-\frac{\nu s(t-s)}{t}|\xi|^2\right).
\end{align}
Because $s/2 \le s(t-s)/t$ for $s\in [0,t/2]$, by symmetry, the above double integral is equal to
\begin{align*}
 2\int_0^{t/2} \ud s \int_{\R^d}  \hat{f}(\ud \xi)\exp\left(-\frac{\nu s(t-s)}{t}|\xi|^2\right)
 &\le 2 \int_0^{t/2} \ud s \int_{\R^d}  \hat{f}(\ud \xi)\exp\left(-\frac{\nu s}{2}|\xi|^2\right)\\
&= 2(2\pi)^d\int_0^{t/2} k(s) \ud s =2 (2\pi)^d h_1(t/2)\\
&\le 2 (2\pi)^d h_1(t),
\end{align*}
where in the last step we have applied Lemma \ref{L:hnInc}.
The induction step is routine. This proves \eqref{E:calL-hUp} and hence \eqref{E:CalK-Up}.

As for the lower bound, we first prove the case $n=1$.
Because  $f$ is nonnegative and
\begin{align}\notag
G\left(\frac{2s(t-s)}{t},z+y-\frac{s}{t}(x-x')\right)
&\ge 2^{-\frac{d}{2}} G\left(\frac{s(t-s)}{t},z+y\right) T_\nu(t,x-x')\\
\notag
&= 2^{-\frac{d}{2}} (2\pi\nu s(t-s)/t)^{-d/2}e^{-\frac{|z+y|^2}{2\nu s(t-s)/s}} T_\nu(t,x-x')\\
\notag
&\ge 2^{-\frac{d}{2}} (2\pi\nu s(t-s)/t)^{-d/2}e^{-\frac{|z|^2+|y|^2}{\nu s(t-s)/s}} T_\nu(t,x-x')\\
\notag
&= 2^{-d} (\pi\nu s(t-s)/t)^{-d/2}e^{-\frac{|z|^2}{\nu s(t-s)/s}} e^{-\frac{|y|^2}{\nu s(t-s)/s}} T_\nu(t,x-x')\\
\notag
&= 2^{-d} G\left(\frac{s(t-s)}{2t},z\right) e^{-\frac{|y|^2}{\nu s(t-s)/s}} T_\nu(t,x-x')\\
\notag
&\ge 2^{-d} G\left(\frac{s(t-s)}{2t},z\right) e^{-\frac{|y|^2}{\nu s/2}} T_\nu(t,x-x')\\
&= 2^{-d} G\left(\frac{s(t-s)}{2t},z\right) T_\nu(t,x-x')T_{\nu/2}(s,y),
\label{E:GDeltaX}
\end{align}
where we have used the fact that $s(t-s)/t \ge s/2$, which is equivalent to $s\in [0,t/2]$.
we see that from \eqref{E:GGf} and \eqref{E:FhatG},
\begin{align*}
\calL_1(t,x,x';y)\ge &
2^{-d} G(t,x) G(t,x') T_\nu(t,x-x')\\
&\times \int_{0}^{t/2}\ud s \: T_{\nu/2}(s,y) \int_{\R^d} \ud z\: f(z)
G\left(\frac{s(t-s)}{2t},z\right)\\
\ge &
2^{-d} G(t,x) G(t,x') T_\nu(t,x-x')\\
&\times \int_{0}^{t/2}\ud s \: T_{\nu/4}(s,y) \int_{\R^d} \ud z\: T_{\nu/2}(t-s,z)f(z)
G\left(\frac{s(t-s)}{2t},z\right).
\end{align*}
Because the function $z\mapsto f(z) T_{\nu/2}(t-s,z)$ is a valid correlation function, i.e., it is symmetric, nonnegative, and nonnegative-definite,
by taking Fourier transform and since $s(t-s)/(2t)\le s/2$, one can see that
\begin{align}\notag
\int_{\R^d} \ud z\: T_{\nu/2}(t-s,z) f(z)  G\left(\frac{s(t-s)}{2t},z\right)
& \ge
\int_{\R^d} \ud z\: T_{\nu/2}(t-s,z) f(z)  G\left(s/2,z\right)\\ \notag
&\ge  3^{-d/2} \int_{\R^d} \ud z\: f(z)  G(s/6,z)\\
&\ge  3^{-d/2} \int_{\R^d} \ud z\: f(z)  G(s,z)\ge 3^{-d/2} k(s),
\label{E_:TfG}
\end{align}
Hence,
\begin{align*}
 \calL_1(t,x,x';y)\ge
2^{-d}3^{-d/2} G(t,x) G(t,x') T_\nu(t,x-x')\int_{0}^{t/2} T_{\nu/4}(s,y) k(s) \ud s,
\end{align*}
where the integral is equal to $h_1(t/2,y)$.
Therefore, the case $n=1$ is true.

Assume that \eqref{E:calL-hLow} is true up to $n$. Then
\begin{align}\notag
 \calL_{n+1}(t,x,x';y)= & \left(\calL_0\rhd\calL_n\right)(t,x,x';y)\\ \notag
\ge& (2\sqrt{3}\:)^{-nd}  \int_0^{\frac{t}{2}}\ud s  \:h_n\left(\frac{t-s}{2},y\right)
\iint_{\R^{2d}}\ud z\ud z' \; G(s,z)G(s,z')\\ \notag
&\times G(t-s,x-z)G(t-s,x'-z') \\ \notag
&\times T_\nu(t-s,(x-z)-(x'-z')) f(y-[(x-z)-(x'-z')])\\ \notag
\ge &
(2\sqrt{3}\:)^{-nd} \int_0^{\frac{t}{2}}\ud s  \:h_n\left(t/2-s,y\right)
\iint_{\R^{2d}}\ud z\ud z' \; G(s,x-z)G(s,x'-z')\\ \notag
&\times G(t-s,z)G(t-s,z')  T_\nu(t-s,z-z') f(y-[z-z'])\\ \notag
= &
(2\sqrt{3}\:)^{-nd} \int_{\frac{t}{2}}^t \ud r  \:h_n\left(r-t/2,y\right)
\iint_{\R^{2d}}\ud z\ud z' \; G(r,z)G(r,z')\\
&\times   G(t-r,x-z)G(t-r,x'-z') T_\nu(r,z-z') f(y-[z-z'])
,\label{E_:Ln+1}
\end{align}
where we have used the fact that $s\mapsto h_n(s,y)$ is nondecreasing (Lemma \ref{L:hnInc}).
Notice that
\[
T_\nu(r,z-z') \ge T_{\nu/2}(r,y-(z-z'))T_{\nu/2}(r,y).
\]
By the same arguments as those in \eqref{E:GGf} and \eqref{E:FhatG}
with the correlation function $f(z)$ replaced by $z\mapsto f(z) T_{\nu/2}(r,z)$, the double integral
$\ud z\ud z'$ in \eqref{E_:Ln+1} becomes
\[
G(t,x)G(t,x') T_{\nu/2}(r,y) \int_{\R^d} \ud z\: T_{\nu/2}(r,z) f(z)  G\left(\frac{2r(t-r)}{t},z+y-\frac{r}{t}(x-x')\right).
\]
By \eqref{E:GDeltaX}, the above quantity is bounded from below by
\[
2^{-d}G(t,x)G(t,x') T_{\nu}(t,x-x') T_{\nu/4}(r,y)
\int_{\R^d} \ud z\: T_{\nu/2}(r,z) f(z)  G\left(\frac{r(t-r)}{2t},z\right),
\]
where we have used the fact that $T_{\nu/2}(r,y)^2=T_{\nu/4}(r,y)$.
Then apply \eqref{E_:TfG} with $s$ replaced by $t-r$ to get
\begin{align*}
 \calL_{n+1}(t,x,x';y)= &
(2\sqrt{3}\:)^{-(n+1)d} G(t,x)G(t,x')T_\nu(t,x-x')\\
&\times \int_{t/2}^t \ud s \: h_n(r-t/2,y) T_{\nu/4}(r,y) k(t-r).
\end{align*}
Because $r\in \:]t/2,t[\:$, $T_{\nu/4}(r,y)\ge T_{\nu/4}(t-r,y)$. Hence,
\begin{align*}
\int_{t/2}^t \ud s \: h_n(r-t/2,y) T_{\nu/4}(r,y) k(t-r)
&\ge
\int_{t/2}^t \ud s \: h_n(r-t/2,y) T_{\nu/4}(t-r,y) k(t-r)\\
&=\int_0^{t/2} h_n(t/2-s,y)T_{\nu/4}(s,y) k(s),
\end{align*}
where the integral is equal to $h_{n+1}(t/2,y)$. This proves the case $n+1$ and \eqref{E:calL-hLow}.
Finally, \eqref{E:CalK-Low} is a direct consequence of \eqref{E:calL-hLow}.
This completes the whole proof of Lemma \ref{L:calK}.
\end{proof}

\begin{lemma}\label{L:InDt}
For all $\mu\in\calM_H(\R^d)$ and all $t\ge 0$, $x,\: x'\in\R^d$,
\begin{align}\label{E:CalKJ1}
\left(\calK\rhd J_1\right)(t,x,x';0) =& \lambda^{-2} \iint_{\R^{2d}}\mu(\ud y)\mu(\ud y')
\calK(t,x-y,x'-y';y'-y) -J_1(t,x,x').
\end{align}
\end{lemma}
\begin{proof}
We first prove \eqref{E:CalKJ1}.
Writing $J_0(t,z)$ and $J_0(t,z')$ in the integral forms and applying the arguments in the proof of Lemma \ref{L:calK},
we see that
\begin{align*}
 \left(\calK\rhd J_1\right)(t,x,x';0) =&
\int_0^t \ud s \iint_{\R^{2d}}\ud z\ud z' \iint_{\R^{2d}}\mu(\ud y)\mu(\ud y')\:  G(s,z-y)G(s,z'-y')\\
&\times f(z'-z) \calK(t-s,x-z,x'-z';z'-z).
\end{align*}
By change of variables, $\hat{z}=z-y$ and $\hat{z}'=z'-y'$,  and by Fubini's theorem,
\begin{align*}
 \left(\calK\rhd J_1\right)(t,x,x';0) =&
\iint_{\R^{2d}}\mu(\ud y)\mu(\ud y') \int_0^t \ud s \iint_{\R^{2d}}\ud \hat{z}\ud \hat{z}' \:   f((y'-y)-(\hat{z}-\hat{z}'))  \\
&\times G(s,\hat{z})G(s,\hat{z}') \calK(t-s,x-y-\hat{z},x'-y'-\hat{z}';(y'-y)- (\hat{z}-\hat{z}'))\\
=&\iint_{\R^{2d}}\mu(\ud y)\mu(\ud y')
\left(\calK\rhd\calL_0\right)(t,x-y,x'-y';y'-y).
\end{align*}
Then use the recursion $\calK\rhd \calL_0 = \lambda^{-2}\calK - \calL_0$ to get \eqref{E:CalKJ1}.
\end{proof}

\bigskip
\begin{proof}[Proof of Theorem \ref{T:ExUni}]
The proof follows the same six steps as those in the proof of
\cite[Theorem 2.4]{ChenDalang13Heat} with some minor changes:

(1) Both proofs rely on estimates on the kernel function $\calK$.
Instead of an explicit formula as for the heat equation case (see
\cite[Proposition 2.2]{ChenDalang13Heat}), Lemma \ref{L:calK} ensures the finiteness and provides a bound on the kernel function $\calK$.

(2) In the Picard iteration scheme (Steps 1--4 in the proof of \cite[Theorem 2.4]{ChenDalang13Heat}),
we need to check the $L^p(\Omega)$-continuity of the stochastic integral,
which then guarantees that at the next step, the integrand is again in $\calP_2$, via \cite[Proposition 3.4]{ChenDalang13Heat}.
The statement of \cite[Proposition 3.4]{ChenDalang13Heat} is still true for $G(t,x)$ on $\R^d$; see \cite[Proposition 2.3.13]{LeChen13Thesis}.
Note that during each iteration, the measurability is guaranteed by Proposition \ref{P:Pm}
(in place of \cite[Proposition 3.1]{ChenDalang13Heat})

(3) In the first step of the Picard iteration scheme, the following property is useful:
For all compact sets $K\subseteq \R_+\times\R^d$,
\[
\sup_{(t,x)\in K}\left(\calK\rhd \left[1+J_1\right] \right) (t,x,x;0)<+\infty.
\]
For the heat equation, this property is discussed in \cite[Lemma 3.9]{ChenDalang13Heat}.
Here, Lemma \ref{L:InDt} gives the desired result with minimal requirements on the initial data.
This property, together with the calculation of  the upper bound on the function $\calK$ in Lemma \ref{L:calK}, guarantees that
all the $L^p(\Omega)$-moments of $u(t,x)$ are finite.
This property is also used to establish uniform convergence of the Picard iteration scheme, hence $L^p(\Omega)$--continuity of $(t,x)\mapsto I(t,x)$.

(4) The moment formula \eqref{E:SecMom} is clear from the Picard iterations.
The formula \eqref{E:SecMom2} is due to Lemma \ref{L:InDt}.

As for \eqref{E:SecMom-Up}, we only need to consider the nonlinear case.
By \eqref{E:WeakCom}, the function $g(t,x,x')=\E\left[u(t,x)u(t,x')\right]$ satisfies
\eqref{E:Recursion} with ``$=$", $\lambda$ and $\calK$ replaced by
``$\le$" $\Lip_\rho$ and $\overline{\calK}$, respectively.

Similarly, for the lower bound \eqref{E:SecMom-Lower},
thanks to \eqref{E:lip}, the above $g$ function satisfies
\eqref{E:Recursion} with ``$=$''  and $\lambda$ replaced by ``$\ge$'' and $\lip_\rho$, respectively.
Hence, this integral inequality is solved by \eqref{E:SecMom-Lower}, i.e., by \eqref{E:SecMom2}
with ``$=$'' and $\lambda$ replaced by ``$\ge$" and $\lip_\rho$, respectively.

This completes the proof of Theorem \ref{T:ExUni}.
\end{proof}

\section{Conditions for phase transitions: proof of Theorem \ref{T:Phase}}\label{S:Phase}
In this section, we will prove Theorem \ref{T:Phase}. We need some lemmas.
%
Lemma \ref{L:LowInd} will be used at the end of the proof of Theorem 1.3.
\begin{lemma}\label{L:LowInd}
Fix $a>0$. Let $c=\nu \pi /(2a^2)$. Then for all $t\in\R_+\times\R^d$,
\[
\int_{[-a,a]^d} G(t,y) \ud y  \ge \left(1+c \: t\right)^{-d/2},
\]
and
\[
\int_0^t\ud s \int_{[-a,a]^d} G(s,y) \ud y \ge
\begin{cases}
2c^{-1}\left(\sqrt{c  t+1}-1\right) & \text{if $d=1$,}\\
c^{-1}\log (1+c\: t)& \text{if $d=2$,}\\
2 \left[c (d-2) \right]^{-1}\left(1-(1+c\: t)^{1-d/2} \right)& \text{if $d\ge 3$.}
\end{cases}
\]
\end{lemma}
\begin{proof}
We only need to prove the case where $d=1$. Notice that
\[
\int_{-a}^a G(t,y)\ud y= 2\Phi\left(\frac{a}{\sqrt{\nu t}}\right) -1,
\]
where $\Phi(x)$ is the distribution of the standard normal distribution.
Denote
\[
F(t):= \sqrt{1+\frac{\nu \pi}{2 a^2} t} \left[2 \Phi\left(\frac{a}{\sqrt{\nu t}}\right) -1\right].
\]
Clearly, $F(0)=1$. By l'Hospital's rule, $\lim_{t\rightarrow\infty} F(t)=1$.
By studying $F'(t)$, one can show that for some $t_0>0$, $F(t)$ is nondecreasing over $[0,t_0]$
and nonincreasing over $[t_0,\infty]$.  Therefore, $F(t)\ge 1$. The rest calculations follow Example \ref{Eg:OU}.
\end{proof}

%

\begin{lemma}
 For all $y\in\R^d$, we have that
\[
\lim_{t\rightarrow\infty} h_1(t)<\infty \quad\Longleftrightarrow\quad
\lim_{t\rightarrow\infty} h_1(t,y)<\infty.
\]
\end{lemma}
\begin{proof}
 Because $h_1(t,y)\le h_1(t)$, the ``if'' part is clear.
On the other hand, for any $\epsilon\in \:]0,t[\:$,
 \begin{align*}
h_1(t,y) & \ge \int_\epsilon^t \ud s \: k(s) T_{\nu/4}(s,y)
\ge T_{\nu/4}(\epsilon,y)\left[h_1(t)-h_1(\epsilon)\right].
 \end{align*}
This proves the lemma.
\end{proof}

Define
\[
H_\nu^*(t,y;\gamma):= \sum_{n=0}^\infty \gamma^n h_1(t/n,y)^n.
\]
\begin{lemma}\label{L:h1LowBd2}
(1) For all $t\ge 0$, $y\in\R^d$ and $\gamma>0$, $H_\nu(t,y;\gamma) \ge H_\nu^*(t,y;\gamma)$.
\\
(2) For all $a>0$ and $y\in\R^d$, if $\gamma\ge e/h_1(a,y)$, then
\[
H_\nu^*(t,y;\gamma)
\ge \frac{e^{t/a}-1}{e-1},\quad \text{for all $t\ge 0$.}
\]
(3) If $\lim_{t\rightarrow\infty} h_1(t)=\infty$, then for all $\gamma>0$ and all $y\in\R^d$,
we have that
\[
H_\nu^*(t,y;\gamma)
\ge \frac{e^{t/a}-1}{e-1},\quad \text{for all $t\ge 0$,}
\]
where $a>0$ is the value such that $h_1(a,y)=e/\gamma$.
\end{lemma}
\begin{proof}
(1) This is because $h_n(t,y)\ge h_1(t/n,y)^n$ for $n\in\bbN$, which is true by induction.\\
(2) Fix $a>0$ and $y\in\R^d$. Note that $h_1(t,y)$ is nondecreasing.
So when $h_1(a,y)>e/\gamma$,
\[
  \sum_{n=0}^\infty \gamma^n h_1(t/n,y)^n \ge\sum_{n=0}^{t/a} \gamma^n h_1(t/n,y)^n
 \ge
\sum_{n=0}^{t/a} \gamma^n h_1(a,y)^n
\ge \frac{e^{\Floor{t/a}+1}-1}{e-1}
\ge \frac{e^{t/a}-1}{e-1}.
\]
(3) Fix arbitrary $\gamma >0$ and $y\in\R^d$.
One can find $a>0$ such that $h_1(a,y)=e/\gamma$.
Then apply the same arguments as those in (2).
\end{proof}

\bigskip

\begin{proof}[Proof of Theorem \ref{T:Phase}]
Fix $\mu\ge 0$.
We first note that if both $\rho_1$ and $\rho_2$ satisfy \eqref{E:lip} and $\rho_1(x)\le \rho_2(x)$ for all $x\in\R$,
then the second moments of the corresponding solutions $u_1(t,x)$ and $u_2(t,x)$ both starting from $\mu$
satisfy the following comparison relation:
\[
\Norm{u_1(t,x)}_2\le \Norm{u_2(t,x)}_2.
\]
Note that $\Norm{\cdot}_p$ denotes the $L^p(\Omega)$-norm.
This is clear from \Itos isometry \eqref{E:Recursion}. Hence, in the following, we need only consider the linear case: $\sigma(x)=\lambda x$ with $\lambda>0$.

We start with the case  where $\Upsilon(0)<\infty$. From \eqref{E:L1}, we know that
\begin{align*}
\calL_1(t,x,x';y)&\le
 (2\pi)^{-d}\: G(t,x)G(t,x')
\int_0^\infty \ud s \int_{\R^d}  \hat{f}(\ud \xi)\exp\left(-\frac{\nu s}{4}|\xi|^2\right)\\
&=\frac{4 }{ \nu (2\pi)^{d}}  G(t,x)G(t,x') \int_{\R^d} \frac{\hat{f}(\ud \xi)}{|\xi|^2} .
\end{align*}
Denote $\theta:=\frac{4 }{ \nu (2\pi)^{d}} \Upsilon(0)$.  Hence, by induction,
\[
\calL_n(t,x,x';y)\le \theta^n G(t,x)G(t,x')
\]
and if $\lambda^2\theta< 1$, i.e.,
\[
\lambda\le 2^{-1}(2\pi)^{d/2}\nu^{1/2} \Upsilon(0)^{-1/2}  =: \underline{\lambda}_{c},
\]
then
\[
\calK(t,x,x';y)\le G(t,x)G(t,x') \frac{1}{1-\theta\lambda^2}.
\]
By \eqref{E:SecMom-Up}, for all $\mu\in\calM_H(\R^d)$ with $\mu\ge 0$,
\[
\Norm{u(t,x)}_2^2 \le J_0^2(t,x)\frac{1}{1-\theta\lambda^2}.
\]
Since $\mu$ satisfies \eqref{E:InitGrow}, for all $\beta>0$,
\[
J_0(t,x)\le \left(\sup_{y\in\R^d} G(t,x-y)e^{\beta|y|}\right)\int_{\R^d} e^{-\beta|y|}\mu(\ud y).
\]
Notice that
\begin{align*}
G(t,x-y)e^{\beta|y|}
& \le \prod_{i=1}^d \frac{1}{\sqrt{2\pi\nu t}}\exp\left(-\frac{(x_i-y_i)^2}{2\nu t} + \frac{\beta}{\sqrt{d}} |y_i|\right)\\
&= (2\pi\nu t)^{-d/2}\exp\left(\frac{\beta^2\nu }{2} t+ \frac{\beta}{\sqrt{d}}\sum_{i=1}^d |x_i|\right).
\end{align*}
Therefore,
\[
\sup_{x\in\R^d}\overline{m}_2(x)\le \beta^2 \nu \quad\text{for all $\beta>0$,}
\]
which implies that $\sup_{x\in\R^d}\overline{m}_2(x)=0$.
On the other hand, part (2) of Lemma \ref{L:h1LowBd2} shows that when $\lambda$ is sufficiently large,
then $\inf_{x\in\R^d}\underline{m}_2(x)>0$.

When $\Upsilon(0)=\infty$, the moment bound \eqref{E:SecMom}, Lemma \ref{L:calK} and part (3) of Lemma \ref{L:h1LowBd2}
together imply that $\inf_{x\in\R^d}\underline{m}_2(x)>0$.
The statement $m_1(x)\equiv 0$ is due to \eqref{E:WeakCom}.

The equivalence between \eqref{E:iff1} and \eqref{E:iff3} is due to \eqref{E:h1Upsilon}.
The implication ``\eqref{E:iff2}$\Rightarrow$\eqref{E:iff3}'' is because that
\[
\lim_{t\rightarrow\infty}h_1(t) = (2\pi)^{-d} \int_0^\infty \ud t \int_{\R^d} \hat{f}(\ud \xi) \exp\left(-\frac{\nu t}{2}|\xi|^2\right)
=\frac{2}{\nu (2\pi)^d} \int_{\R^d}\frac{\hat{f}(\ud \xi)}{|\xi|^2}.
\]
On the other hand, if $d\le 2$, then Lemma \ref{L:LowInd} implies that \eqref{E:iff3} fails.
This proves the implication ``\eqref{E:iff3}$\Rightarrow$\eqref{E:iff2}''.
This completes the whole proof of Theorem \ref{T:Phase}.
\end{proof}

\section{Intermittency fronts: proof of Theorem \ref{T:Front}}\label{S:Front}
We first consider the following lemma which will be used in the proof of Theorem \ref{T:Front}:
\begin{lemma}\label{L:InitBeta}
If $\mu\ge 0$ satisfies \eqref{E:muBeta} for some $\beta>0$,  then we have
\[
J_0^2(t,x)\le  C^2 (2\pi\nu t)^{-d} \exp\left(-\frac{2\beta}{\sqrt{d}}|x| +\nu \beta^2 t\right),
\]where $C=\int_{\R^d}e^{\beta|x|}\mu(\ud x)$.
\end{lemma}
\begin{proof}
Notice that
\[
\frac{|y_1|+\cdots+|y_d|}{\sqrt{d}}\le |y|=\sqrt{y_1^2+\cdots+y_d^2}\le |y_1|+\cdots+|y_d|.
\]
By  the same arguments as the proof of \cite[Lemma 4.4]{ChenDalang13Heat} with $\beta$ replaced by $\beta/\sqrt{d}$,
\[
J_0^2(t,x)\le C^2 (2\pi\nu t)^{-d}  \prod_{i=1}^d \exp\left(-\frac{2\beta}{\sqrt{d}}|x_i|+\frac{\nu\beta^2}{d} t\right)
\le C^2 (2\pi\nu t)^{-d} \exp\left(-\frac{2\beta}{\sqrt{d}}|x| +\nu \beta^2 t\right).
\]
\end{proof}

\begin{proof}[Proof of Theorem \ref{T:Front}]
We first prove the upper bound. By \eqref{E:SecMom2} and \eqref{E:CalK-Up},
\[
\Norm{u(t,x)}_2^2\le  \Lip_\rho^{-2} J_0^2(t,x) \exp\left(\theta t\right),
\]
where $\theta:=\theta(\nu,\Lip_\rho)$ is defined in \eqref{E:Var}. Hence, by Lemma \ref{L:InitBeta}, for $\alpha>0$,
\[
\sup_{|x|>\alpha t} \Norm{u(t,x)}_2^2
\le \Lip_\rho^{-2} (2\pi\nu t)^{-d} \exp\left(-\frac{2\beta}{\sqrt{d}}\alpha t +\nu \beta^2 t+\theta t\right),
\]
where $C:=\int_{\R^d}e^{\beta|x|}\mu(\ud x)$. Now, the exponential growth rate are
\[
-\frac{2\beta}{\sqrt{d}}\alpha t+\nu \beta^2 t +\theta t<0\quad\Longleftrightarrow\quad \alpha> \frac{\sqrt{d}}{2}\left(\nu \beta +\frac{\theta}{\beta}\right),
\]
which proves the upper bound.

Now we consider the lower bound.
Denote $\kappa:=(2\sqrt{3}\:)^{-d}$.
By \eqref{E:SecMom-Lower} and \eqref{E:CalK-Low},
\begin{align*}
 \Norm{u(t,x)}_2^2\ge \lip_\rho^{-2}
\iint_{\R^{2d}}\mu(\ud z)\mu(\ud z')\;
\calL_0(t,x-z,x-z') T_\nu(t,z-z') H_\nu\left(t/2,z-z';\kappa\lip_\rho^{2}\right).
\end{align*}
Fix a constant $a>0$ such that $\int_{[-a,a]^d}\mu(\ud z)>0$.
Denote $\vec{a}=(a,\cdots,a)\in\R^d$.
For $z$ and $z'\in [-a,a]^d$, we have that
\[
T_\nu(t,z-z')\ge T_{\nu}(t,2\vec{a}),
\]
and
\[
H_\nu\left(t/2,z-z';\kappa\lip_\rho^2\right)\ge
H_\nu\left(t/2,2\vec{a};\kappa\lip_\rho^2\right).
\]
Notice that
\begin{align*}
\calL_0(t,x-z,x-z') &\ge (2\pi\nu t)^{-d}
\exp\left(-\frac{2|x|^2+|z|^2+|z'|^2}{\nu t}\right)\\
&=2^{-d}\calL_0(t/2,z,z') T_{\nu}(t/2, x).
\end{align*}
Thus,
\[
\Norm{u(t,x)}_2^2\ge \lip_\rho^{-2}
C_t^2\;
T_\nu(t/2,x) H_\nu\left(t/2,2\vec{a};\kappa\lip_\rho^2\right).
\]
where
\[
C_t=\int_{[-a,a]^d}G(t/2,z)\mu(\ud z).
\]
Hence, for $\alpha >0$,
\[
\sup_{|x|\ge \alpha t}\Norm{u(t,x)}_2^2\ge
\lip_\rho^{-2}C_t^2\;
\exp\left(-\frac{2\alpha^2 t}{\nu}\right) H_\nu\left(t/2,2\vec{a};\kappa\lip_\rho^2\right),
\]
and
\[
\mathop{\lim\inf}_{t\rightarrow+\infty}
\frac{1}{t}
\sup_{|x|\ge \alpha t}\log \Norm{u(t,x)}_2^2\ge
-\frac{2\alpha^2}{\nu}+
\mathop{\lim\inf}_{t\rightarrow\infty}\frac{1}{t}\log H_\nu\left(t/2,2\vec{a};\kappa\lip_\rho^2\right).
\]
Therefore, by part (1) of Lemma \ref{L:h1LowBd2},
\[
\underline{\lambda}(2)
\ge
\left(
\nu
\mathop{\lim\inf}_{t\rightarrow\infty}\frac{1}{t}
\log H_\nu^*\left(t,2\vec{a};\kappa\lip_\rho^2\right)\right)^{1/2}.
\]
Then apply Lemma \ref{L:h1LowBd2} for the above limit.
This completes the proof of Theorem \ref{T:Front}.
\end{proof}

\appendix
\section{Appendix: Examples}\label{SA:ST-Con}
\begin{example}(Riesz kernels)
Suppose $f(z)=|z|^{-\alpha}$ with $\alpha\in \:]0,2\wedge d[\:$. Then
\begin{align*}
k(t) &= (2\pi\nu t)^{-d/2}\int_0^{\infty} \exp\left(-\frac{r^2}{2\nu t}\right) r^{-\alpha+d-1} \frac{\pi^{d/2}}{\Gamma(1+d/2)}\ud r = C_{\alpha,d} \:t^{-\alpha/2},
\end{align*}
with $C_{\alpha,\nu,d}:=\frac{\nu^{-\alpha /2}}{2^{1+\alpha/2}}
\frac{\Gamma\left((d-\alpha)/2\right)}{\Gamma\left(1+d/2\right)}$, and
\[
h_1(t)= C_{\alpha,\nu,d}^*
\:t^{1-\alpha/2},\qquad \Upsilon(\beta)=C_{\alpha,\nu,d}' \beta^{-1+\alpha/2},
\]
for some constants $C_{\alpha,\nu,d}^*=\frac{\nu^{-\alpha /2}}{2^{\alpha/2}(2-\alpha)}
\frac{\Gamma\left((d-\alpha)/2\right)}{\Gamma\left(1+d/2\right)}$ and $C_{\alpha,\nu,d}'>0$.
By induction,
\[
h_n(t) = C_{\alpha,\nu,d}^n \frac{t^{n(1-\alpha/2)}\Gamma(1-\alpha/2)^n}{\Gamma(n(1-\alpha/2)+1)},\quad\text{for all $n\ge 0$},
\]
and hence
\[
H_\nu(t;\lambda^2)=E_{1-\alpha/2,1}\left(\lambda^2 C_{\alpha,\nu,d} \: \Gamma(1-\alpha/2) t^{1-\alpha/2}\right),
\]
where $E_{\alpha,\beta}(z)$ is the {\it Mittag-Leffler} function with two parameters
\[
E_{\alpha,\beta}(z):=\sum_{n=0}^\infty \frac{z^n}{\Gamma( \alpha n+\beta)},\quad \Re \alpha>0, \;\beta\in \mathbb{C}, \; z\in\mathbb{C};
\]
see, e.g., \cite{Podlubny99FDE}.
The following asymptotic expansions are useful: As $|z|\rightarrow\infty$,
\begin{align}\label{E:MLA}
E_{\alpha,\beta}(z)\sim \frac{1}{\alpha}\exp\left(z^{1/\alpha}\right)-\sum_{k=1}^\infty \frac{z^{-k}}{\Gamma(\beta-\alpha k)},\quad \text{if  $0<\alpha<2$ and $|\arg z|<\alpha \pi/2$,}
\end{align}
we have that
\[
\lim_{t\rightarrow\infty}\frac{1}{t}\log H_\nu(t;\lambda^2)= \left[C_{\alpha,\nu,d} \: \Gamma(1-\alpha/2)\right]^{\frac{2}{2-\alpha}} \: \lambda^{\frac{4}{2-\alpha}} .
\]
By Lemma \ref{L:h1LowBd2},
\[
\mathop{\lim\inf}_{t\rightarrow\infty}\frac{1}{t}\log H_\nu^*(t;\lambda^2) \ge C_{\alpha,\nu,d}^\star \lambda^{\frac{4}{2-\alpha}},\qquad
\text{with}\quad
C_{\alpha,\nu,d}^\star=\left(\frac{C_{\alpha,\nu,d}^*}{e}\right)^{\frac{2}{2-\alpha}}.
\]
\end{example}

\begin{example}(Ornstein-Uhlenbeck-type kernels)\label{Eg:OU}
Suppose $f(z)=\exp\left(-|z|^\alpha \right)$ for $\alpha \in \:]0,2]$.
The case when $\alpha=2$ has closed forms:
\[
k(t)= (2\pi\nu t)^{-d/2}\int_0^t \exp\left(-\frac{r^2}{2\nu t}-r^2\right) r^{+d-1} \frac{\pi^{d/2}}{\Gamma(1+d/2)}\ud r
=d^{-1}(1+2 \nu  t)^{-d/2},
\]
and
\[
h_1(t)=
\begin{cases}
\nu^{-1}\left(\sqrt{2 \nu  t+1}-1\right) & \text{if $d=1$,}\\
(4\nu)^{-1}\log(1+2 \nu  t)& \text{if $d=2$,}\\
\left[\nu (d-2) d \right]^{-1}\left(1-(1+2 \nu t)^{1-d/2} \right)& \text{if $d\ge 3$,}
\end{cases}
\]
and
\[
\Upsilon(\beta)=(2\pi)^{-d}\int_{\R^d}\frac{\pi^{d/2} e^{-|\xi|^2/2}}{\beta+|\xi|^2}\ud \xi
=d^{-1}  2^{-d} e^{\beta /2} \beta
   ^{\frac{d}{2}-1} \Gamma
   \left(1-\frac{d}{2},\frac{\beta
   }{2}\right),
\]
where $\Gamma(\nu, x):=\int_x^\infty t^{\nu-1}e^{-t}\ud t$ is the {\it incomplete Gamma function}.
\end{example}


\begin{example}(Brownian motion case)
When $f(z)\equiv 1$, the noise reduces to a space-independent noise. In this case,
\[
k(t)\equiv 1,\qquad h_1(t)=t,\qquad \Upsilon(\beta)=(2\pi)^{-d}\beta^{-1},
\]
and by \eqref{E:Var} and  Lemma \ref{L:h1LowBd2},
\[
\lim_{t\rightarrow\infty}\frac{1}{t}\log H_\nu(t;\lambda^2)=\frac{\lambda^2}{(2\pi)^d},
\qquad
\mathop{\lim\inf}_{t\rightarrow\infty}\frac{1}{t}\log H_\nu^*(t;\lambda^2)\ge\frac{\lambda^2}{e}.
\]
\end{example}

\begin{example}(Space-time white noise case)\label{Eg:WhiteNoise}
 When $d=1$ and $f=\delta_0$, we have that
\[
k(t)=\frac{1}{\sqrt{2\pi\nu t}},\qquad h_1(t)=\sqrt{\frac{2 t}{\pi \nu}},\qquad \Upsilon(\beta)=\frac{1}{2\sqrt{\beta}},
\]
and by \eqref{E:Var} and  Lemma \ref{L:h1LowBd2},
\[
\lim_{t\rightarrow\infty}\frac{1}{t}\log H_\nu(t;\lambda^2)=\frac{\lambda^4}{2\nu},
\qquad
\mathop{\lim\inf}_{t\rightarrow\infty}\frac{1}{t}\log H_\nu^*(t;\lambda^2)\ge\frac{2\lambda^4}{\pi\nu e^2}.
\]
\end{example}

\begin{example}(Lower bound for $d=1,2$)
When $f(x)\ge 1_{[-a,a]^d}(x)$ for some $a>0$ and $d=1,2$, then by Lemma \ref{L:LowInd} and \ref{L:h1LowBd2},
\[
\mathop{\lim\inf}_{t\rightarrow\infty}\frac{1}{t}\log H_\nu^*(t;\lambda^2)\ge \frac{\nu \pi}{2a^2}\left[\left(1+\frac{\nu\pi e}{4a^2 \lambda^2}\right)^2-1\right]^{-1}\rightarrow
\frac{\lambda^2}{e}
\quad \text{as $\lambda\rightarrow\infty$ if $d=1$,}
\]
and
\[
\mathop{\lim\inf}_{t\rightarrow\infty}\frac{1}{t}\log H_\nu^*(t;\lambda^2)\ge \frac{\nu \pi}{2a^2}\left[\exp\left(\frac{\nu\pi e}{2a^2 \lambda^2}\right)-1\right]^{-1}\rightarrow
\frac{\lambda^2}{e} \quad \text{as $\lambda\rightarrow\infty$ if $d=2$.}
\]
\end{example}

\section{Associative property of the convolution ``$\rhd$''}

\begin{lemma}\label{L:Associative}
Let $h$, $w$, and $g$ be three real-valued functions defined on $\R_+\times\R^{3d}$.
Suppose that
$\left( h \rhd \left(w\rhd g\right) \right)(t,x,x';y)$ and
$\left( \left(h\rhd w\right)\rhd g \right)(t,x,x';y)$ are well defined where $t\ge 0$, $x$, $x'$ and $y\in\R^d$. Then
\[
\left( h \rhd \left(w\rhd g\right) \right)(t,x,x';y)
=
\left( \left(h\rhd w\right)\rhd g \right)(t,x,x';y).
\]
\end{lemma}
\begin{proof} By definition,
\begin{align*}
\left( h \rhd \left(w\rhd g\right) \right)&(t,x,x';y)\\
=&
\int_0^t\ud s_1 \iint_{\R^{2d}}\ud z_1\ud z_1' \: h(t-s_1,x-z_1,x'-z'_1;y-(z_1-z'_1))\\
&\times \left(w\rhd g\right)(s_1,z_1,z'_1;y) f(y-(z_1-z'_1))\\
=&
\int_0^t\ud s_1 \iint_{\R^{2d}}\ud z_1 \ud z'_1 \: h(t-s_1,x-z_1,x'-z'_1;y-(z_1-z'_1))f(y-(z_1-z'_1)) \\
&\times \int_0^{s_1}\ud s_2 \iint_{\R^{2d}} \ud z_2 \ud z'_2\;
w(s_1-s_2,z_1-z_2,z_1'-z_2';y-(z_2-z_2'))\\
&\times g(s_2,z_2,z_2';y)f(y-(z_2-z_2'))
\end{align*}
Then by change of variables
\begin{align*}
 \hat{s}_1=t-s_2 && \hat{z}_1=x-z_2 && \hat{z}_1'=x'-z_2'\;,
 \\
 \hat{s}_2=t-s_1 && \hat{z}_2=x-z_1 && \hat{z}_2'=x'-z_1'\;,
\end{align*}
we see that
\begin{align*}
\left( h \rhd \left(w\rhd g\right) \right)&(t,x,x';y)\\
=&
\int_0^t\ud \hat{s}_1 \iint_{\R^{2d}}\ud \hat{z}_1 \ud \hat{z}'_1 \:
g(t-\hat{s}_1,x-\hat{z}_1,x'-\hat{z}'_1;y)f(y-[(x-\hat{z}_1)-(x-\hat{z}'_1)]) \\
&\times \int_0^{\hat{s}_1}\ud \hat{s}_2 \iint_{\R^{2d}} \ud \hat{z}_2 \ud \hat{z}'_2\;
w(\hat{s}_1-\hat{s}_2,\hat{z}_1-\hat{z}_2,\hat{z}_1'-\hat{z}_2';y-[(x-\hat{z}_1)-(x'-\hat{z}_1')])\\
&\times h(\hat{s}_2,\hat{z}_2,\hat{z}_2';y-[(x-\hat{z}_2)-(x'-\hat{z}_2')])f(y-[(x-\hat{z}_2)-(x'-\hat{z}_2')])\\
=&
\int_0^t\ud \hat{s}_1 \iint_{\R^{2d}}\ud \hat{z}_1 \ud \hat{z}'_1 \:
g(t-\hat{s}_1,x-\hat{z}_1,x'-\hat{z}'_1;y)f(y-[(x-\hat{z}_1)-(x-\hat{z}'_1)]) \\
&\times \left(h\rhd w\right)(\hat{s}_1,\hat{z}_1,\hat{z}_1';y-[(x-\hat{z}_1)-(x-\hat{z}'_1)])\\
=& \left(\left(h\rhd w\right)\rhd g\right)(t,x,x';y).
\end{align*}
This completes the proof of Lemma \ref{L:Associative}.
\end{proof}

\section*{Acknowledgements}
\addcontentsline{toc}{section}{Acknowledgements}
This work was mostly done when both authors were at the University of Utah in the year of 2014.
Both authors appreciate many useful comments from and discussions with {\it Davar Khoshnevisan}.
The first author thanks {\it Daniel Conus} for many stimulating discussions on this problem when
the first author was invited to give a colloquium talk at Lehigh university in Oct. 2014.
The non-symmetric operator ``$\rhd$'' was studied in the first author's thesis.
He would like to express his gratitude to {\it Robert C. Dalang} for his careful reading of his thesis.

\begin{small}
\bigskip

\begin{minipage}[t]{0.5\textwidth}
\noindent\textbf{Le Chen},\\
\noindent Department of Mathematics, \\
\noindent University of Kansas,\\
\noindent 405 Snow Hall, 1460 Jayhawk Blvd\\
\noindent Lawrence, Kansas, 66045-7594\\
\noindent\emph{Email}: \texttt{chenle@ku.edu}\\
\noindent\emph{URL}: \url{http://www.math.ku.edu/u/chenle}
\end{minipage}
%
\begin{minipage}[t]{0.5\textwidth}
\noindent\textbf{Kunwoo Kim}\\
\noindent Mathematical Sciences Research Institute\\
\noindent 17 Gauss Way, \\
\noindent Berkeley, CA 94720\\
\noindent\emph{Email}: \texttt{kunwookim@msri.org}\\
\noindent\emph{URL}: \url{http://www.math.utah.edu/~kkim}
\end{minipage}
\end{small}

\end{document}